\documentclass[12pt]{amsart}

\usepackage{hyperref}
\hypersetup{colorlinks=true,linkcolor=blue,urlcolor=red,citecolor=red}
\usepackage{setspace}
\usepackage{amsmath}
\usepackage{amsthm}

\newtheorem{theorem}{Theorem}
\newtheorem{remark}{Remark}
\newtheorem*{future work}{Future work}
\usepackage{pdfpages}
\usepackage{amssymb,latexsym}
\newcommand{\beq}{\begin{equation}}
\newcommand{\eeq}{\end{equation}}
\newtheorem{lemma}{Lemma}

\newtheorem{property}{Property}

\newcommand{\RNum}[1]{\uppercase\expandafter{\romannumeral #1\relax}}

\newtheorem{corollary}{Corollary}
\begin{document}
\begin{center}
{\Large \textbf{Generalised Barred Preferential Arrangements}}
\end{center}

\noindent\begin{minipage}[t]{7cm}\fontsize{13}{1}
Sithembele Nkonkobe,  Venkat Murali,
\fontsize{8}{1}
{\it Department of Mathematics (Pure \&Applied)\\ Rhodes
University, Grahamstown, 6140,\\ South Africa,\\ snkonkobe@yahoo.com, v.murali@ru.ac.za}
\end{minipage}
\hspace{5mm}
\begin{minipage}[t]{7cm}
Be\'{a}ta B\'{e}nyi\\\fontsize{8}{1}
{\it  Faculty of Water Sciences\\ National University of Public Service\\ Hungary,\\ beata.benyi@gmail.com}
\end{minipage}

\section*{Abstract\label{section:1}}
 \noindent A barred preferential arrangement is a preferential arrangement onto which a number of identical bars are inserted in between the blocks  of the preferential arrangement. In this study we examine combinatorial properties of barred preferential arrangements whose elements are colored with a number of available colors. 
 
  \noindent\fontsize{10}{1}\selectfont     
Mathematics Subject Classifications: 05A18, 05A19, 05A16 
\vspace{0.1mm}
\\\textbf{Keyword(s)}: Preferential arrangements, barred preferential arrangements, generating functions of Nelsen-Schmidt type. 
\normalsize
 
\begin{section}{Introduction}

\noindent\textbf{Barred preferential arrangements.}

\noindent A \textit{preferential arrangement} of a set $X_n=\{1,2,3,\ldots,n\}$ is an ordered partition of the set $X_n$.
Introducing $\xi$ bars (where $\xi\in\mathbb{N}_0=\{0,1,2,3,\ldots,\}$) in between blocks of a preferential arrangement results in a \textit{barred preferential arrangement}. The $\xi$ bars induce $\xi+1$ sections in which the elements of $X_n$ are preferentially arranged (see~\cite{barred:2013,Nkonkobe & Murali Nelesn-Schmidt generating function}). The following are some examples of barred preferential arrangements of the set $X_6$ having two bars and three bars respectively.

$a)$ $35\quad 2|\quad|1\quad4\quad6$

$b)$ $|5\quad136\quad4\quad2|\quad |$

The two bars in $a)$ give rise to three sections; namely, the first section (from left to right) has two blocks, the second section is empty i.e the section between the two bars, and the third section has three blocks. Similarly, the barred preferential arrangement in $b)$ has four sections of which three are empty. Barred preferential arrangements seem to first appear in \cite{barred:2013}.
\begin{remark} In this study barred preferential arrangements are viewed as being formed by first placing bars and then preferentially arranging elements on the sections formed. The authors in \cite{Nkonkobe & Murali Nelesn-Schmidt generating function} have used the same technique in their arguments.
\end{remark}

  


\noindent In \cite{Nelsensconj} Nelsen conjectured that;  \begin{equation}\label{equation:1}\sum\limits_{k=0}^{n}\sum\limits_{s=0}^{k}\binom{k}{s}(-1)^{k-s}(\gamma+s)^n=\frac{1}{2}\sum\limits_{s=0}^{\infty}\frac{(\gamma+s)^n}{2^{s}},\end{equation} for $\gamma\in\mathbb{R}$ (the set of real numbers) and non-negative integer $n$.
In \cite{SoluNelsensconj} three alternative proofs of the conjecture were given by Donald Knuth et al. Here we propose and prove an alternative identity to the ones given by  Donald Knuth et al in \cite{SoluNelsensconj} in  generalising Nelsen's conjecture and show how one side our generalised identity can be interpreted combinatorially in terms of barred preferential arrangements. From now on we will refer to \eqref{equation:1} as Nelsen's Theorem. 

The authors in \cite{Higher order geometric polynomials} proposed the following generating function as a way of generalising geometric polynomials, $\frac{(1+\alpha x)^{\gamma/\alpha}}{(1-y((1+\alpha x)^{\beta/\alpha}-1))^{\lambda}}$. The  authors using a non combinatorial argument recognised that the generating function offers a form of generalised barred preferential arrangements. In this study we examine combinatorial properties of  integer sequences arising from the generating function when $\alpha=0$ and $y=1$.
 \normalsize 
\end{section}

   \begin{section}{Nelsen-Schmidt Question}

    Nelsen and Schmidt in \cite{Chains in power set Nelsen 91} proposed the family of generating functions,\begin{equation}\label{equation:8} \frac{e^{\gamma x}}{2-e^x}.\end{equation}  In the manuscript for $\gamma=2$ the authors interpreted the generating function as being that of the number of chains in the power set of $X_n$. The generating function for $\gamma=0$ is known to be that of number of preferential arrangements/number of outcomes in races with ties (see~\cite{gross:1962,mendelson:1982}). In the manuscript Nelsen and Schmidt then asked, ``could there be combinatorial structures associated with either $X_n$ or the power set of $X_n$ whose integer sequences are generated by  members of the family in \eqref{equation:8} for other values of $\gamma$?" We will now refer to this question as the Nelsen-Schmidt question, and to the generating function in \eqref{equation:8} as the Nelsen-Schmidt generating function. 
   In answering the question of Nelsen and Schmidt the authors in \cite{Nkonkobe & Murali Nelesn-Schmidt generating function} have proposed that the generating function $\frac{e^{\gamma x}}{2-e^{x}}$ is that of number of restricted barred preferential arrangements for all values of $\gamma$ in non-negative integers. Furthermore, the authors in \cite{Nkonkobe & Murali Nelesn-Schmidt generating function} interpreted the following more general generating function in terms of restricted barred preferential arrangements, 
   \begin{equation}\label{equation:3}\frac{e^{\gamma x}}{(2-e^x)^\lambda},\end{equation}
   
   where $\gamma$ and $\lambda$ are non-negative integers not simultaneously equalling to 0. 
   
  In the following we examine combinatorial properties of a further generalisation of the Nelsen-Schmidt generating function, 
   
   \begin{equation}\label{equation:4}
   \frac{e^{\gamma x}}{2-e^{\beta x}},
   \end{equation}
   where $\beta,\gamma$ are non-negative integers not simultaneously equalling to 0.

     \noindent In the following the numbers $S(n,i,\alpha,\beta,\gamma)$ are generalised stirling numbers. The numbers seem to first appear in \cite{A unified approach to generalized Stirling numbers}.
      \begin{lemma}\cite{Corcino Rorberto 2001}\label{lemma:3} For $\alpha,\beta,\gamma\in\mathbb{N}_0$, where $(\alpha,\beta,\gamma)\not=(0,0,0)$, \fontsize{10}{1}
      
      $S(n,i,\alpha,\beta,\gamma)=\frac{1}{\beta^ii!}\Delta^i(\beta i+\gamma|\alpha)_n\big|_{s=0}=\frac{1}{\beta^ii!}\sum\limits_{s}^{}(-1)^{i-s}\binom{i}{s}(\beta s+\gamma|\alpha)_n$.
      \end{lemma}
      \begin{theorem} \label{Theorem:15}
       \label{theorem:10}\cite{Paper with interpretation of G-stirling numbers}
     Suppose $\alpha,\beta,\gamma$ are non-negative integers such that $\alpha$ divides both $\beta$, and $\gamma$. Given $i+1$ distinct cells such that the first $i$ cells each contains $\beta$ labelled compartments, and the $(i+1){th}$ cell contains $\gamma$ labelled compartments. The compartments are given cyclic ordered numbering. The compartments are limited to one ball. The number $\beta^i i!S(n,i,\alpha,\beta,\gamma)$, is the number of ways of distributing $n$ distinct elements into the $i+1$ cells such that the first $i$ cells are non-empty.
      \end{theorem}
      
     Using Lemma~\ref{lemma:3} we obtain the following corollary of Theorem~\ref{theorem:10}.
      
         \begin{corollary}\label{lemma:6} Given $X_n$. Partitioning $X_n$ into $i+1$ distinct blocks, such that the first $i$ blocks have $\beta$ labelled compartments, and the $(i+1)^{th}$ section has $\gamma$ labelled compartments. The number 
              $i!\beta^iS(n,i,0,\beta,\gamma)$ is the number of all possible partitions of $X_n$ into the $i+1$ blocks such that only the $(i+1){th}$ block may be empty.  
               \end{corollary}
      
      The numbers  $B_n(\alpha,\beta,\gamma)$ are defined in the following ways. We will refer to these numbers later.
      \begin{lemma}\cite{On generalised Bell polynomials}For $\alpha,\beta,\gamma\in\mathbb{N}_0$, where  where $(\alpha,\beta,\gamma)\not=(0,0,0)$,\label{lemma:5}
      $$B_n(\alpha,\beta,\gamma)=\sum\limits_{i}^{}i!\beta^iS(n,i,\alpha,\beta,\gamma).$$
      \end{lemma}
     
      \begin{lemma}\cite{On generalised Bell polynomials}For real/complex $\alpha,\beta,\gamma$ such that $(\alpha,\beta,\gamma)\not=(0,0,0)$,
             \begin{equation}\label{equation:21}
              \sum\limits_{n=0}^{\infty}B_n(\alpha,\beta,\gamma)\frac{x^n}{n!}=\frac{(1+\alpha t)^{\gamma/\alpha}}{2-(1+\alpha t)^{\beta/\gamma}}.
             \end{equation}
                                  
             \end{lemma}

         
         
         
          \begin{property}\label{property:1}
         
         By Lemma~\ref{lemma:3} we have,
         \begin{equation}\label{equation:11}
               S(n,i,0,\beta,0)=\frac{1}{\beta^ii!}\sum\limits_{s}^{}(-1)^{i-s}\binom{i}{s}(\beta s)^n.\end{equation} 
        Hence, for fixed $i,\beta$ in non-negative integers, the number $\beta^ii!S(n,i,0,\beta,0)$ is the number of ways of partitioning an $n$-element set into $i$ blocks where each of the $i$ blocks has $\beta$ labelled compartments such that none of the blocks is empty.

         \noindent $A$: $X_n$ set is partitioned into $i$ ordered blocks. 
          
          \noindent$B$: Elements on each block are distributed into $\beta$ labelled\\\indent\indent compartments.

        \noindent The question is; how many partitions of $X_n$ are possible satisfying $A$\\\indent and $B$.
        
       \noindent The above question can be rephrased in the following way. 
        
       \noindent In how many ways can we distribute $n$ distinct elements into $i$ labelled cells where each of the cells has $\beta$ compartments, such that none of the $i$ cells is empty, where $0\leq i\leq n$. 
         By \eqref{equation:11} and Lemma~\ref{lemma:5} the number we are looking for is   
          $B_n(0,\beta,0)$.  
         \end{property}

    \begin{property}\label{property:2}
$X_n$ is distributed into $\gamma$ labelled compartments.   
    \end{property}

    \noindent Denote $\begin{bmatrix}
       \frac{x^n}{n!}
       \end{bmatrix}\frac{e^{\gamma x}}{2-e^{\beta x}}$ by $H_{n}(\beta,\gamma)$.

         \begin{remark}
         Throughout the remainder of this paper in forming barred preferential arrangements(BPAs) where applicable, the section with property~\ref{property:2} will be the first section from the left(sometimes this section will be referred to as the special section) the remaining sections will all have property~\ref{property:1} unless stated otherwise.
         \end{remark} 
         
         In the following theorem we answer the Nelsen-Schmidt question asked in \eqref{equation:8} in a generalised form.
   \begin{theorem}\label{theorem:7}
   The generating function $\frac{e^{\gamma x}}{2-e^{\beta x}}$ for $\beta,\gamma$ in non-negative integers (where $(\beta,\gamma)\not=(0,0))$, is that of the number of barred preferential arrangements with one bar, such that one section has property~\ref{property:1} and the other section has property~\ref{property:2}.
   \end{theorem}
   \begin{proof}
     By \eqref{equation:4} we have,
         \begin{equation}\label{equation:9}
            H_n(\beta,\gamma)=\sum\limits_{r}\binom{n}{r}B_{r}(0,\beta,0)\gamma^{n-r}.
            \end{equation}
            
        Hence, $H_n(\beta,\gamma)$ is the number of barred preferential arrangements of an $n$-element set having one bar such that one fixed section has property~\ref{property:1} above and the other section has property~\ref{property:2}. 
        \end{proof}
        \begin{remark}
        The generating function $\frac{e^{\gamma x}}{2-e^{\beta x}}$ can be derived from the generating function in \eqref{equation:21}.
        \end{remark}

   \begin{theorem}\label{theorem:3}$\beta\geq0$ and $n,\gamma\geq1$,
   $$H_n(\beta,\gamma)=\gamma^n+\sum\limits_{i=0}^{n-1}\binom{n}{i}H_i(\beta,\gamma)\beta^{n-i}.$$
   \end{theorem}
   \begin{proof}
   Let $\mathcal{H}_n(\beta, \gamma)$ denote the set of barred preferential arrangements on $n$ elements with one bar (so with two sections), such that the elements of the left hand side are labelled further with a number between $\{1,\ldots, \gamma\}$, while the elements right to the bar with a number from the set $\{1,\ldots, \beta\}$. Clearly, $|\mathcal{H}_n(\beta, \gamma)|=H_n(\beta, \gamma)$. We obtain an element of the set $h\in\mathcal{H}_n(\beta, \gamma)$ the following way: if there is no element on the right hand side of the bar, then we need only to assign to each element of $[n]$ a number from $[\gamma]$, which gives $\gamma^n$ possibilities. If there is at least one element to the right of the bar, then first let us construct the block right next to the bar in this section from $(n-i)$ elements in $\binom{n}{n-i}\beta^{n-i}$ ways. The remainder $i$ elements form an element $h_i$ of  $\mathcal{H}_i(\beta, \gamma)$. Since $n-i\not=0$, we obtain the number of all elements in $\mathcal{H}_n(\beta, \gamma)$ by summing up over $i$, where $i$ runs from $0$ to  $n-1$. 
   \end{proof}
  
   \begin{theorem}\label{theorem:4}For $n,\beta,\gamma\geq0,$  where $(\beta,\gamma)\not=(0,0),$
   
   $$H_{n+1}(\beta, \gamma)=\gamma H_n(\beta,\gamma)+\beta\sum\limits_{i}^{}\binom{n}{i}H_i(\beta,\gamma)H_{n-i}(\beta,\beta).$$
   \end{theorem}
   \begin{proof}
   The recursion is based on the process of the inserting the $(n+1)$th element into a barred preferential arrangement on $n$ elements. First, we can insert the $(n+1)$th element into the block of the left section. Then, we just need to choose a label from $[\gamma]$ for this new element. Otherwise, let $B^*$ denote the block into which we add $(n+1)$. Cut the barred preferential arrangement before $B^*$ and let $i$ denote the number of elements in the part before $B^*$. The first part is then a barred preferential arrangement from $\mathcal{H}_i(\beta, \gamma)$, while the second part can be seen also as a barred preferential arrangement of the rest of the elements  with $B^*$ as the special block next to left of the bar, i.e.,  from $\mathcal{H}_{n-i}(\beta, \beta)$.  We choose in $\binom{n}{i}$ ways the elements for the first part, and choose the label of $(n+1)$ in $\beta$ ways. Multiplying these together and summing up by letting the index $i$ to run, we obtain the theorem. 
   \end{proof}
   \begin{theorem}\label{theorem:5} For $n,\beta,\gamma\geq0$, $(\beta,\gamma)\not=(0,0),$
   $$B_n(0,\beta,0)=\sum_{i}^{}\binom{n}{i}H_i(\beta,\gamma)(-1)^{n-i}\gamma^{n-i}.$$
   
   \end{theorem}

   \begin{proof}
   Let $\mathcal{B}_i$ be the number of barred preferential arrangement of the set $\mathcal{H}_n(\beta, \gamma)$ with at least $(n-i)$ elements in the first, special block with $\gamma$ compartments. $|\mathcal{B}_i|=\binom{n}{n-i}H_i(\beta,\gamma)\gamma^{n-i}$. The application of the inclusion- exclusion principle completes the proof.
   \end{proof}
  
\end{section}
   \begin{section}{generalised barred preferential arrangements.}

The generating function for the number of barred preferential arrangements is (see\cite{barred:2013});  \begin{equation}\label{equation:13}\frac{1}{(2-e^{x})^{\lambda+1}},\end{equation} where $\lambda$ is a non-negative integer. The work in \cite{Nkonkobe & Murali Nelesn-Schmidt generating function} is one generalisation of these barred preferential arrangements where the authors studied the generating function
  
  \begin{equation}\label{equation:5}
  \frac{e^{\gamma x }}{(2-e^{x})^{\lambda}},
  \end{equation}
  
  where $\lambda,\gamma\in\mathbb{N}_0$, where $\mathbb{N}_0$ is the set of non-negative integers.
   
   Another generalisation of barred preferential arrangements has been done in \cite{Higher order geometric polynomials}.

     In this section we propose a further generalisation of barred preferential arrangements to the generalisation we have in the previous section by interpreting the following generating function in terms of barred preferential arrangements,
   \begin{equation}\label{equation:2}\frac{e^{\gamma x}}{(2-e^{\beta x})^{\lambda}},\end{equation}  where $\gamma\in\mathbb{N}_0$, and $\lambda,\beta\in\mathbb{N}$  (positive integers). 
   

  \begin{lemma}\label{lemma:2}\cite{On generalised Bell polynomials}For real/complex $\alpha,\beta,\gamma$ such that $(\alpha,\beta,\gamma)\not=(0,0,0)$,
        $$B_n(\alpha,\beta,\gamma)=\frac{1}{2}\sum\limits_{k=0}^{\infty}\frac{(\beta k+\gamma|\alpha)_n}{2^k}.$$
        \end{lemma}

    Denote $\begin{bmatrix}
   \frac{x^n}{n!}
   \end{bmatrix}\frac{e^{\gamma x}}{(2-e^{\beta x})^{\lambda}}$ by $H_{n}(\lambda,\beta,\gamma)$.

   \begin{theorem}\label{theorem:6}
   Given $\lambda,\gamma\in\mathbb{N}_0$ such that $(\lambda,\gamma)\not=(0,0)$, and $\beta\in\mathbb{N}$,
   
   
   $H_n(\lambda,\beta,\gamma)$ is the number of barred preferential arrangements of $X_n$ such that $\lambda$ of the sections have property~\ref{property:1} and one section has property~\ref{property:2}.
   \end{theorem}
   \begin{proof}

      $\frac{1}{2-e^{\beta x}}=\frac{1}{2}\sum\limits_{k=0}^{\infty}\frac{e^{(\beta k)x}}{2^k}\implies\begin{bmatrix}
      \frac{x^n}{n!}
      \end{bmatrix}\frac{1}{2-e^{\beta x}}=\frac{1}{2}\sum\limits_{k=0}^{\infty}\frac{(\beta k)^n}{2^k}$. 
    
     By $H_{n}(\lambda,\beta,\gamma)=\begin{bmatrix}\frac{x^n}{n!}
        \end{bmatrix}\frac{e^{\gamma x}}{(2-e^{\beta x})^{\lambda}}$ and Lemma~\ref{lemma:2} we have, 
       \begin{equation}\label{equation:12}H_n(\lambda,\beta,\gamma)=\sum\limits_{r_1+\cdots+r_{\lambda+1} = n}\binom{n}{r_1,r_2,\ldots,r_{\lambda+1}}\:\gamma^{r_1}\prod \limits_{i=2}^{\lambda+1} B_{r_{i}}(0,\beta,0).\end{equation}
   

      \end{proof}
      Note the generating function $\frac{e^{\gamma x}}{(2-e^{\beta x})^{\lambda}}$ is a special case of the generating function $\frac{(1+\alpha x)^{\gamma/\alpha}}{(1-y((1+\alpha x)^{\beta/\alpha}-1))^{\lambda}}$ studied in \cite{Higher order geometric polynomials}.
      \begin{remark}

       The interpretation of the numbers $H_n(\lambda,\beta,\gamma)$ given in Theorem~\ref{theorem:6} does two things;  
      
      \noindent \RNum{1}). It is a generalisation of the barred preferential arrangements in \eqref{equation:13} and \eqref{equation:5}. 
      
      \noindent \RNum{2}). Further, it further generalises the answer to the Nelsen Schmidt question we have given in Theorem~\ref{theorem:7} above.\end{remark}  
      
      \begin{remark}
           
            The special case $\lambda=1,\beta=2$, and $\gamma=0$ on Theorem~\ref{theorem:6} is the sequence A216794 in \cite{oeis}.
           \end{remark}
           \begin{remark}
          The statement of Theorem~\ref{theorem:6} for the special case $\lambda=1$ can be derived from that given for the numbers $B_n(\alpha,\beta,\gamma)$ in \cite{On generalised Bell polynomials}.
           \end{remark} 
    \begin{theorem}\label{theorem:8}
      \fontsize{10}{1}For $\beta\in\mathbb{N}$, and $\lambda\geq2$, 
      \begin{equation}\label{equation:14}
       H_n(\lambda,\beta,\beta)=\frac{1}{2}H_n(\lambda-1,\beta,\beta)+\frac{1}{2\beta(\lambda-1)}\sum\limits_{i=0}^{n}\binom{n}{i}H_{i+1}(\lambda-1,\beta,0)\beta^{n-i}.
      \end{equation}
      \end{theorem}
      \begin{proof}
      First, we write the formula in a combinatorially nicer form. 
      \fontsize{10}{1}
\[2\beta(\lambda -1)H_n(\lambda,\beta,\beta)=\beta(\lambda-1)H_n(\lambda-1, \beta,\beta)+\sum_{i=0}^ {n}\binom{n}{i}H_{i+1}(\lambda-1, \beta, 0)\beta^{n-i}\]
Consider the set of elements of $\mathcal{H}_n(\lambda, \beta, \beta)$ such that one of the $\beta$ compartments is colored red and one of the $\lambda$ bars, except the first one, is marked with a $0$ or a $1$. We let $\mathcal{H}^*_n(\lambda, \beta, \beta)$ denote the set of the so obtained decorated barred preferential arrangements. The left hand side of the equality enumerates this set. We describe a map, that associates to each decorated preferential arrangement of  $\mathcal{H}^*_n(\lambda, \beta, \beta)$ another barred preferential arrangement so that the image of the map is a set enumerated by the right hand side. Consider the label of the chosen bar.  If the bar has a $0$, delete the bar and insert a block with a single extra $(n+1)$th element. 
If the bar is labelled by $1$, consider what is right next to the left of the bar. If there is a block, insert $(n+1)$ into this block, if it is another bar, delete this bar. In each cases when inserting $(n+1)$, it is also colored red, i.e., receives the same $\beta$-compartment that is chosen.  The number of barred preferential arrangements that we obtain by deleting a bar, (and not inserting $(n+1)$) is $\beta(\lambda -1)H_n(\lambda-1,\beta, \beta)$, since one $\beta$-compartment is still colored red, and we have only $\lambda-1$  with a $1$ marked bar left. In the other cases, we obtain a barred preferential arrangement on $n+1$ elements, i.e., elements of the set $\mathcal{H}_{n+1}(\lambda-1,\beta,\beta)$, such that the first special section does not contain the $(n+1)$th element. This is, because the first bar was not marked, hence during the insertion process $(n+1)$ was never put into the section left to the first bar. 
The number of these barred preferential elements is 
$\sum_{i=0}^{n}\binom{n}{i}H_{i+1}(\lambda-1, \beta, 0)\beta^{n-i}$. We obtain this formula according to the enumeration of the following pairs: choose the $n-i$ elements for the first special block and construct it in $\beta^{n-i}$ ways. Combine these blocks with barred preferential arrangements on $(i+1)$ elements with $\lambda-1$ sections and empty first, special section, for which we have $H_{i+1}(\lambda-1, \beta, 0)$ possibilities. 
      \end{proof}
   Theorem~\ref{theorem:8} is a generalisation of Theorem~1 of \cite{barred:2013}.
   

\begin{theorem} For $\gamma,\lambda\in\mathbb{N}_0$, where $(\lambda,\gamma)\not=(0,0)$,
\fontsize{10}{1}
\begin{equation}
 H_{n+1}(\lambda,\beta,\gamma)=\gamma H_n(\lambda,\beta,\gamma)+\lambda\beta\sum_{i=0}^{n}\binom{n}{i}H_i(1,\beta,\beta)H_{n-i}(\lambda,\beta,\gamma). 
\end{equation}
\end{theorem}
   \begin{proof}
  Let $\mathcal{H}_n(\lambda, \beta, \gamma)$ denote the set of barred preferential arrangements with $\lambda$ bars (so $\lambda+1$ sections), such that the first section includes one special block with elements labelled from the set $\{1,\ldots, \gamma\}$, and the elements in the rest of the blocks labelled from the set $\{1,\ldots, \beta\}$. We enumerate the set $\mathcal{H}_{n+1}(\lambda, \beta, \gamma)$  based on the position of the element $(n+1)$. It can be included in the special first block, which gives $\gamma H_n(\lambda,\beta,\gamma)$ possibilities Otherwise, let $B^*$ be the block that contains $(n+1)$. Consider the portion of the barred preferential arrangement from $B^*$ till the next bar to its right (including $B^*$ itself), and let $i$ be the number of elements contained in these blocks. This portion can be seen as a barred preferential arrangement from the set $\mathcal{H}_i(\beta,\beta)=\mathcal{H}_i(1,\beta, \beta)$. Ignoring this portion of the barred preferential arrangements, the remaining elements form a barred preferential arrangements from $H_{n-i}(\lambda, \beta, \gamma)$. For this construction we need to choose the $i$ elements out of the $n$ elements in $\binom{n}{i}$ ways, the section in that $B^*$ is placed in $\lambda$ ways and finally, the label of $(n+1)$ in $\beta$ ways. Multiplying these together and summing up completes the argument. 
   \end{proof}

\begin{theorem}\label{theo:old11}For $\gamma,\lambda\in\mathbb{N}_0$, where $(\lambda,\gamma)\not=(0,0)$,
\begin{equation}
H_{n+1}(\lambda,\beta, \gamma)=\gamma H_n(\lambda,\beta, \gamma )+\lambda\beta H_n(\lambda +1, \beta, \gamma+\beta).
\end{equation}
\end{theorem}
\begin{proof}
Again, the left hand side is the size of the set $\mathcal{H}_{n+1}(\lambda, \beta, \gamma)$. Consider the $(n+1)$th element. If it is contained in the first section, (let's denote this block by $\Gamma$), then there are $\gamma H_n(\lambda, \beta,  \gamma)$ possibilities to obtain such a barred preferential arrangement on $n+1$ elements from a one on $n$ elements. 
Assume now that the $(n+1)$th  element is in a block, say $B^*$, with $\beta$ compartments. Decompose the section including $B^*$ as $B_1B^*B_2$, where $B_1$ and $B_2$ are ordered partitions with the extra structure of having a label for each element from $[\beta]$ on each block. We reorder the parts of this barred preferential arrangement as follows: Move the block $B^*$ to the left of the first block, and merge $\Gamma$ and $B^*$ into one block. Insert instead of the block $B^*$ a bar between the sequences of blocks $B_1$ and $B_2$, and finally, delete $(n+1)$. We obtain this way a barred preferential arrangement on $n$ elements, with $(\lambda +1)$ bars and $(\gamma+\beta)$ compartments in the first, special block. Hence, the number of such barred preferential arrangements is $H_n(\lambda+1, \beta, \gamma+\beta)$. There are two information that we have to keep in track: which $\beta$ compartment was the $(n+1)$th element assigned to, and which bar is the inserted bar. Hence, we have $\lambda\beta H_n(\lambda +1, \beta, \gamma+\beta)$ as total number of barred preferential arrangements on $n+1$ elements such that the $(n+1)$th element is not in the first, special block. 
\end{proof}

   \begin{theorem}\label{theorem:9}For $\gamma,\lambda\in\mathbb{N}$,
   \begin{equation}\label{equation:15}
   H_n(\lambda,\beta,\gamma+\beta)=2H_n(\lambda,\beta,\gamma)-H_n(\lambda-1,\beta,\gamma).
   \end{equation}
   \end{theorem}
   \begin{proof}
   Consider the set $\mathcal{H}_n(\lambda, \beta, \gamma+\beta)$. In these barred preferential arrangements the elements in the first block are labelled from the set  $\{1,2,\ldots, \gamma, \gamma+1, \ldots, \gamma+\beta\}$. The number of such barred preferential arrangements that have only labels from the set $\{1, \ldots, \gamma\}$ is $H_{n}(\lambda, \beta, \gamma)$. If there is at least one element with a label from $\{\gamma +1, \ldots, \gamma+\beta\}$, then move these elements to the right of the first bar, to create the first block in the ordered partition of the second section. We obtain this way a barred preferential arrangement with $\gamma$ compartments in the first, special block and at least one block in the second section with $\beta$ compartments. How many such barred preferential are there? $H_n(\lambda, \beta, \gamma)-H_{n}(\lambda-1, \beta, \gamma)$, since we need to exclude the barred preferential arrangements that do not have any block in the second section, which are clearly in bijection with barred preferential arrangements with one less, i.e., $(\lambda-1)$ bars.
   \end{proof}
   
   Theorem~\ref{theorem:9} is a generalisation of Theorem~9 of \cite{Nkonkobe & Murali Nelesn-Schmidt generating function}.
   
   \begin{theorem}For $\beta\in\mathbb{N}$,and $\lambda\geq 2$, 
   \begin{equation}
   H_n(\lambda,\beta,0)=\frac{1}{2\beta(\lambda-1)}H_{n+1}(\lambda-1,\beta,0)+\frac{1}{2}H_n(\lambda-1,\beta,0).
   \end{equation}
   \end{theorem}
   \begin{proof}
 This proof is similar to that of Theorem~\ref{theorem:8}. We rewrite the identity as
\[2\beta(\lambda -1)H_n(\lambda, \beta, 0)=\beta(\lambda-1)H_n(\lambda-1, \beta, 0)+H_{n+1}(\lambda-1, \beta,0)\]
The left hand side is the number of decorated barred preferential arrangements of $\mathcal{H}^*_n(\lambda, \beta, 0)$, (with empty first section). Inserting $n+1$ according to the above rule, the deletion of the marked bar without inserting $n+1$ leads to barred preferential arrangements on $n$ with one $\beta$ compartment chosen and one of its $\lambda-1$ bars marked. This gives $\beta(\lambda-1)H_n(\lambda-1, \beta, 0)$ possibilities. Deleting the marked bar and inserting $(n+1)$ leads to barred preferential arrangements on $n+1$ elements, $\lambda -1$ bars, and empty first section, for which we have $H_{n+1}(\lambda-1, \beta, 0)$ possibilities. 
   \end{proof}
   
   \end{section}
   
     The following theorem offers a generalisation of Nelsen's Theorem discussed in \eqref{equation:1}. 
              \begin{theorem}\label{theorem:30}For  $\beta,\gamma,\lambda\in\mathbb{R}$, and $n\in\mathbb{N}_0$, where $(\lambda,\gamma)\not=(0,0)$, \fontsize{10}{1} 
               \begin{equation}\sum\limits_{k=0}^{n}\sum\limits_{s=0}^{k}\binom{k}{s}(-1)^{k-s}H_n(\lambda-1,\beta,\gamma+\beta s)=\sum\limits_{s=0}^{\infty}\frac{ H_n(\lambda-1,\beta,\gamma+\beta s) }{2^{s+1}}.\end{equation} 
              \end{theorem}
          \begin{proof} $\frac{e^{\gamma x}}{(2-e^{\beta x})^\lambda}=\frac{e^{\gamma x}}{(2-e^{\beta x})^{\lambda-1}}\sum\limits_{k=0}^{\infty}(e^{\beta x}-1)^k$. 
        \\$\implies\begin{bmatrix}\frac{x^n}{n!}\end{bmatrix}\frac{e^{\gamma x}}{(2-e^{\beta x})^\lambda}=\sum\limits_{k=0}^{n}\sum\limits_{s=0}^{k}\binom{k}{s}(-1)^{k-s}H_n(\lambda-1,\beta,\gamma+\beta s)$
       \\Also, $\frac{e^{\gamma x}}{(2-e^{\beta x})^\lambda}=\frac{1}{2}\frac{e^{\gamma x}}{(2-e^{\beta x})^{\lambda-1}}\sum\limits_{s=0}^{\infty}\frac{e^{xs\beta}}{2^s}$. 
       \\$\implies\begin{bmatrix}\frac{x^n}{n!}\end{bmatrix}\frac{e^{\gamma x}}{(2-e^{\beta x})^\lambda}=\frac{1}{2}\sum\limits_{s=0}^{\infty}\frac{H_n(\lambda-1,\beta,\gamma+\beta s)}{2^s}$. 
        \end{proof}
        The argument used in Theorem~\ref{theorem:30} is a generalisation of that used in proving Equations 2 and 4 of \cite{gross:1962}.
         
         
         \begin{theorem}\label{theorem:28} For $\gamma,n\in \mathbb{N}_0$ and $\beta,\lambda\in \mathbb{N}$, the number of barred preferential arrangements of $X_n$ having $\lambda$ bars. Where the first section has $\gamma+k\beta$ compartments ($0\leq k\leq n$) and the remaining $\lambda-1$ sections have property~\ref{property:1}, such that none of the $k$ parts each having $\beta$ compartments is empty,
         
         \begin{equation}\label{equation:24}\sum\limits_{k=0}^{n}\sum\limits_{s=0}^{k}(-1)^{k-s}\binom{k}{s}H_n(\lambda-1,\beta,\gamma+\beta s).\end{equation}
         
         \end{theorem}
         \begin{proof} Considering a barred preferential arrangement of $X_n$ having $\gamma+k\beta$ compartments on the first section (where $k=0,1,2,\ldots,n$) and $\lambda-1$ sections having property~\ref{property:1}. The number of those arrangements in which $s$ of the $k$ parts having $\beta$ compartments are empty is\\ $H_n(\lambda-1,\beta,\gamma+\beta k-\beta s)$. The inclusion-exclusion principle completes the proof.
         \end{proof}

         \begin{remark}  Theorem~\ref{theorem:28}  is a generalised combinatorial interpretation of one side of Nelsen's Theorem discussed in \eqref{equation:1}.\nonumber
         \end{remark}
         
           \begin{corollary}\label{corollary:1} For $\gamma\in\mathbb{N}_0$ and $\beta,\lambda\in\mathbb{N}$,
               \begin{equation}\label{equation:25} \sum\limits_{k=0}^n\sum\limits_{s=0}^{k}\binom{k}{s}(-1)^{k-s}H_n(\lambda-1,\beta,\gamma+\beta s)=H_n(\lambda,\beta,\gamma).
               \end{equation}
               \end{corollary}
               \begin{proof} In Theorem~\ref{theorem:28} the $k$ parts with $\beta$ compartments  each on the first section can collectively can be interpreted as a single section having property~\ref{property:1}. The by the inclusion/exclusion principle we have,
             
              \noindent                 
                  $\sum\limits_{k=0}^n\sum\limits_{s=0}^{k}\binom{k}{s}(-1)^{k-s}H_n(\lambda-1,\beta,\gamma+\beta s)=H_n(\lambda,\beta,\gamma)$. 
              
               \end{proof}
               \noindent The proof of Corollary~\ref{corollary:1} is a generalisation of that used in obtaining Equation 2 of~\cite{mendelson:1982}. 
                   \begin{corollary}[Gould \& Mays\cite{GouldandMayspaperonchainsinpowerset}]  The number of chains on the power set of $X_n$ is,  
                           \begin{center}
                           $\sum\limits_{k=0}^{n}\sum\limits_{s=0}^{k}\binom{k}{s}(2+s)^n(-1)^{k-s}$,
                           \end{center} for all $n\in\mathbb{N}_0$.
                           \end{corollary}

The generalisation of the Nelsen- Schmidt type generating functions can be interpreted formally, automatically by the symbolic method\cite{Book Analytic combinatorics}. The general generating function,
\[\frac{e^{\gamma x}}{(2-e^{\beta x})^{\lambda}},\]
is the result of the translation of the following construction according to the general theory:
\[[\mbox{SET}(\mathcal{X})]^{\gamma}\times[\mbox{SEQ}([\mbox{SET}(\mathcal{X})]^{\beta}_{>0}]^{\lambda}.\]

Formally, this is a a triple of the following: a $\gamma$-tuple of possible empty sets, a $\lambda$-tuple of non-empty sequences of $\beta$-tuples of sets of elements. This is clearly equivalent to the set $\mathcal{H}_n(\lambda, \beta, \gamma)$. Namely, the first section is the $\gamma$-tuple of sets. Further, since the order of the sections are determined by the $\lambda$ bars, it an be seen as $\lambda$-tuples of the containments of the sections. Each section contains  an ordered partition with the extra structure of having $\beta$ compartments. This is a sequence of $\beta$-tuples of sets, and here we do not allow that all entries of the $\beta$-tuples are empty.

           \begin{lemma}\cite{generatingfunctionology}\label{lemma:7}
           Assume $f(z)=\sum\limits_{n=0}^{\infty}a_ nz^n$ is an analytic function in some region containing the origin, assume that a singularity of $f(z)$ of smallest modulus be at a point $z_0\not=0$, and suppose that $\epsilon>0$ is given. Then $\exists$ $N$ such that $\forall$ $n>N$ we have $|a_n|<\begin{pmatrix}
           \frac{1}{|z_0|}+\epsilon
           \end{pmatrix}^n$. 
           \end{lemma}
           
           \begin{theorem}For every $\epsilon>0$, and $\lambda,\beta\in\mathbb{N}$, $\exists$ $N\in \mathbb{N}$ such that for $n>N$,
           
           \begin{equation} 
            H_n(\lambda,\beta,\gamma)=\mathcal{O}\normalsize\begin{pmatrix}n!\begin{pmatrix}
                      \frac{\beta}{\log2}+\epsilon
                      \end{pmatrix}^n\end{pmatrix}.\end{equation}
           \end{theorem}
           \begin{proof}
            \begin{equation}\label{equation:20}f(z)=\frac{e^{\gamma z}}{(2-e^{\beta z})^{\lambda}}.\end{equation}

     $f(z)$ has the singularities      
     $z=\frac{\log2+2k\pi i}{\beta}$ where $k\in\mathbb{N}_0$. So $f(z)$ is analytic on the disk $|z|<\frac{\log2}{\beta}$, the special case $\beta=1,\gamma=0$ and $\lambda=1$ is found in page 45-48 of \cite{generatingfunctionology}. By lemma~\ref{lemma:7} we obtain the result.
     
     \end{proof}

\end{document}